\documentclass[11pt]{article}
\usepackage{a4wide}
\usepackage{epsfig}
\usepackage{oldgerm}
\usepackage{amsmath}
\usepackage{amsthm}
\usepackage{amssymb}
\usepackage[numbers,sort&compress]{natbib}
\usepackage{epstopdf}
\usepackage{amstext,color}
\usepackage{array,multirow}
\usepackage{fullpage}
\usepackage{fixmath}
\usepackage{graphicx,graphics}

\usepackage[vlined,ruled,linesnumbered]{algorithm2e}

\newtheorem{theorem}{Theorem}[section]
\newtheorem{proposition}{Proposition}[section]

\newtheorem{lemma}{Lemma}[section]
\newtheorem{remark}{Remark}[section]
\def\CC{{\textmd \kern.24em \vrule width.02em height1.4ex depth-.05ex \kern-.26emC}}
\def\TagOnRight
\def\QQ{\rlap {\raise 0.4ex \hbox{$\scriptscriptstyle |$}} {\hskip -0.1em Q}}
\catcode`\@=11
\begin{document}
\begin{center}
 {{\bf \large {\rm {\bf  Fractional Crank-Nicolson-Galerkin Finite Element Scheme for the Time-Fractional Nonlinear Diffusion Equation}}}}
\end{center}
\begin{center}
{\textmd {\bf Dileep Kumar}}\footnote{\it Department of Mathematics, Indian Institute of Technology, Delhi, India, (dilipkmr832@gmail.com)},\\
{\textmd {{\bf Sudhakar Chaudhary}}}\footnote{\it Department of Mathematics, Jaypee Institute of Information Technology, Noida, Uttar Pradesh, India (sudhakarpatel.iitd@gmail.com)},
{\textmd {{\bf V.V.K Srinivas Kumar}}}\footnote{\it Department of Mathematics, Indian Institute of Technology, Delhi, India, {(vvksrini@maths.iitd.ac.in)}  }
\end{center}

\begin{abstract}
This article presents a finite element scheme with Newton's method for solving the time-fractional nonlinear diffusion equation. For time discretization, we use the fractional Crank-Nicolson scheme based on backward Euler convolution quadrature. We discuss the existence-uniqueness results for the fully-discrete problem. Discrete fractional Gronwall type inequality for the backward Euler convolution quadrature which is used to approximate the Riemann-Liouville fractional derivative is established by using the idea of D.Li et al. given in \cite{D.}.  {\it{A priori}} error estimate for the fully-discrete problem in $L^2(\Omega)$ norm is derived. Numerical results based on finite element scheme are provided to validate theoretical estimates.
\end{abstract}
{\bf Keywords:} 
Time-fractional diffusion equation,  Fractional Crank-Nicolson  method, Error estimates,  Discrete fractional Gronwall type inequality
\section{Introduction} 
Fractional partial differential equations (FPDEs) have been widely applied in several real life problems in biology, engineering, and physics \cite{west2007, Metzler, Kil,Podlubny, Mainardi}. 
In this work, we consider the following time-fractional nonlinear diffusion equation.\\
Find $u=u(x,t),$ $x\in\Omega$ and $t>0$ such that
\begin{subequations}
	\begin{align}
	\label{cuc:1.1}
	^C{D}^{\alpha}_{t}u-\Delta{u} &= f(u)  \quad \mbox{in}  \quad \Omega\times(0,T],\\
	\label{cuc:1.3}
	u(x,t)=&0 \quad \mbox{on}\quad\partial{\Omega}\times(0,T],\\
	\label{cuc:1.4}
	u(x,0)=&0,\hspace{.3cm}  \quad \mbox{in} \quad \Omega,
	\end{align}
\end{subequations}
where $\Omega$ is a bounded convex polygonal/polyhedron domain in $\mathbb{R}^d$ $(d=1,2,3)$ with boundary $\partial\Omega$ and
$ f(u)$ represents the forcing term.
Here $^C{D}^{\alpha}_{t}\varphi$ denotes the $\alpha^{th}$ order $(0<\alpha< 1)$ Caputo fractional derivative   of the function $\varphi(t)$ and it is defined as
\begin{equation}
\begin{split}
^{C}{D}^{\alpha}_{t}\varphi(t):=\frac{1}{\Gamma(1-\alpha)}\int_{0}^{t}(t-s)^{-\alpha}\frac{\partial \varphi(s)}{\partial s}ds,\\
\end{split}
\end{equation}
where $\Gamma(\cdot)$ is the gamma function.\\
We can express Caputo fractional derivatives in terms of Riemann-Liouville fractional derivatives by the following relation \cite{Kil}
\begin{equation}
^CD^{\alpha}_{t}\varphi(t)= \ ^RD^{\alpha}_{t}\big(\varphi(t)-\varphi(0)\big),
\end{equation}
where the Riemann-Liouville fractional derivative $^RD^{\alpha}_{t}\varphi$ is defined as
\begin{equation}
^RD^{\alpha}_{t}\varphi(t) :=\frac{1}{\Gamma(1-\alpha)}\frac{d}{dt}\int_{0}^{t}(t-s)^{-\alpha}\varphi(s)ds.
\end{equation}
Note that under the condition $\varphi(0)=0,$  Caputo and Riemann Liouville fractional derivatives coincide. Since $u(x,0)=0,$ therefore problem \eqref{cuc:1.1}-\eqref{cuc:1.4} can also be considered in terms of Riemann Liouville fractional derivative.
\par In general, it is difficult to obtain analytical solution for most of the FPDEs. Thus one has to look for efficient numerical methods for solving these type of PDEs. Several numerical methods have been proposed in the literature to solve these equations \cite{Lin,diethelm1997algorithm,lubich1986discretized,lubich1988convolution,dimitrov2013}. These numerical methods are broadly divided into two categories namely, L1-type methods and convolution quadrature. Former methods have been introduced in \cite{Lin,diethelm1997algorithm} and are based on piecewise polynomial interpolation while later methods have been developed by Lubich \cite{lubich1986discretized,lubich1988convolution} and these methods are generated by high-order backward difference formulas.\\
Authors in \cite{jiang2011high,ford2011finite} considered the following linear time-fractional diffusion equation
\begin{equation}\label{linear}
\begin{split}
^{C}{D}^{\alpha}_{t}u(x,t)&-\Delta u(x,t)=f(x,t),  \quad (x,t)\in [0,1]\times(0,T],\\
u(0,t)&=u(1,t)=0,  \quad t\in(0,T],\\
u(x,0)&=u_0(x), \quad x\in[0,1].\\
\end{split}
\end{equation}
They used L1-type of methods for the time discretization and Galerkin finite element methods for space discretization.\\
Further, authors in \cite{D.} solved the following nonlinear diffusion equation  by using L1-Galerkin finite element methods
\begin{equation}\label{nfopde1}
\begin{split}
^{C}{D}^{\alpha}_{t}u-\Delta u&=f(u)  \quad \mbox{in}  \quad \Omega\times(0,T],\\
u(x,t)&=0  \quad \mbox{on}  \quad \partial\Omega\times(0,T],\\
u(x,0)&=u_0(x) \quad \mbox{in} \quad \Omega,\\
\end{split}
\end{equation}
where  $\Omega$ is a bounded convex polygonal/polyhedron domain in $\mathbb{R}^d$ $(d=1,2,3)$ and $f:\mathbb{R}\rightarrow\mathbb{R}$ is a Lipschitz continuous function.
In \cite{D.}, authors also derived the discrete fractional Gronwall type inequality which is significant in the analysis of fractional order parabolic partial differential equations.\\
Authors in \cite{jin2016two} employed convolution quadrature for time discretization and Galerkin finite element methods for space discretization to solve the following fractional diffusion $(0<\alpha<1)$ and diffusion-wave $(1<\alpha<2)$ equation with nonsmooth data 
\begin{equation}\label{linear3d}
\begin{split}
^{C}{D}^{\alpha}_{t}u(x,t)-\Delta u(x,t)&=f(x,t)  \quad \mbox{in}  \quad \Omega\times(0,T],\\
u(x,t)&=0  \quad \mbox{on}  \quad \partial\Omega\times(0,T],\\
u(x,0)&=u_0(x) \quad \mbox{in} \quad \Omega,\\
\end{split}
\end{equation}
where $\Omega$ is a bounded convex polygonal domain in $\mathbb{R}^d$ $(d=1,2,3)$ with boundary $\partial\Omega$. Also, in this work \cite{jin2016two} authors have established the optimal error estimates concerning data regularity. \\
In \cite{Jin}, authors studied nonlinear time-fractional diffusion equation \eqref{nfopde1}. In order to solve this problem, authors applied L1 method as well as backward Euler convolution quadrature commonly known as Gr$\ddot{\mathrm{u}}$nwald-Letnikov approximation. Also, fractional Gronwall type inequality has been developed at continuous as well as discrete levels.\\
Recently, authors in \cite{jin2017analysis} analyzed the fractional Crank-Nicolson-Galerkin finite element scheme, developed by Dimitrov \cite{dimitrov2013}, for solving problem \eqref{linear3d} $(0<\alpha<1)$. In this work, authors have shown $O(\Delta t^2)$ accuracy in time for both smooth and nonsmooth  data.\\
Motivated by the above works, in this paper we present fractional Crank-Nicolson-Galerkin finite element scheme for solving the time-fractional nonlinear diffusion equation. To the best of our knowledge, this is the first attempt to use Crank-Nicolson-Galerkin finite element scheme for solving nonlinear time-fractional diffusion equation. The main contributions of this work are summarized as follows:
\begin{itemize}
	\item[1.] We propose a new fractional Crank-Nicolson-Galerkin finite element scheme for solving the nonlinear time-fractional diffusion equation. Newton's method is employed for linearizing the nonlinear discrete problem.
	\item[2.] Authors in \cite{D.} established an important discrete fractional Gronwall type inequality for L1 approximation to the Caputo fractional derivative. Motivated by this work, in this paper we establish a discrete fractional Gronwall type inequality for backward Euler convolution quadrature (Gr$\mathrm{\ddot{u}}$nwald-Letnikov approximation) to the Riemann-Liouville fractional derivative.
	\item[3.] We prove the well-posedness of the fully-discrete scheme and derive {\it{a priori}} error estimate in $L^2(\Omega)$ norm for the discrete problem. These theoretical results are confirmed via several numerical experiments.
\end{itemize}
The remainder of the article is organized as follows: In Section 2, we review the preliminaries and some known results. Section 3 presents a fully-discrete fractional Crank-Nicolson-Galerkin finite element scheme and discusses the existence-uniqueness of the fully-discrete solution. Section 4 establishes fractional Gronwall type inequality and provides {\it{a priori}} error estimate for the fully-discrete problem in $L^2(\Omega)$ norm. Numerical results which substantiate the theoretical estimates are provided in Section 5. Section 6 concludes the paper.
\section{Preliminaries and some known results}
\noindent In this section, we introduce some definitions and function spaces.\\
Let $L^2(\Omega)$ be the inner product space with inner product $\langle\cdot,\cdot\rangle$ defined by $\langle v,w\rangle=\int_{\Omega}v(x)w(x)dx$ and norm $\|v\|=\big(\int_{\Omega}|v(x)|^2 dx\big)^{\frac{1}{2}}$.
For a nonnegative integer $m,$ $H^{m}(\Omega)$ denotes the usual Sobolev space on domain $\Omega$ with the norm
\begin{equation*}
\|w\|_m=\Big(\sum_{0\leq a\leq m}\Big\| \frac{\partial^a{w}}{\partial x^a}\Big\|^2\Big)^\frac{1}{2}.\\
\end{equation*}
Let $C_0^{\infty}(\Omega)$ be the space of infinitely differentiable function with compact support in $\Omega$ and $H^{m}_{0}(\Omega)$ is the closure of $C_0^{\infty}(\Omega)$ with respect to the norm $\|\cdot\|_m$.\\
Throughout, the notation $C$ denotes a generic constant that may vary at different occurrences.\\
Following hypothesis is needed in the existence-uniqueness of the solution and the error analysis.\\
$H$: The function $f:\mathbb{R} \rightarrow\mathbb{R}$ is Lipschitz continuous with $|f(u_1)-f(u_2)|\leq L |u_1-u_2|,$ for $u_1, $ $u_2$ $\in \mathbb{R},$ and $L>0.$\\
In the following, we state the existence-uniqueness and regularity results for the  problem \eqref{cuc:1.1}-\eqref{cuc:1.4}. The proof of these results  can be found in \cite{Jin}.
\begin{theorem}\label{uniquet}
	Under the hypothesis $H,$ the problem \eqref{cuc:1.1}-\eqref{cuc:1.4} admits a unique solution $u$ for $0<\alpha<1,$ such that
	\begin{equation*}
	\begin{split}
	u\in C^{\alpha}\big([0,T];L^2(\Omega)\big)\cap C\big([0,T];H^1_0(\Omega)\cap H^2(\Omega)\big), & \ \ ^CD^{\alpha}_t u \in C\big([0,T];L^2(\Omega)\big),\\
	\frac{\partial u(t)}{\partial t} \in L^2(\Omega) \ \mbox{and}  \ \  \Big\|\frac{\partial u(t)}{\partial t}\Big\|\leq Ct^{\alpha-1}, \ \  \mbox{for} \ \ t\in (0,T].\\
	\end{split}
	\end{equation*}
\end{theorem}
\section{Fractional Crank-Nicolson-Galerkin finite element scheme}
Let $\mathcal{T}_h$ be the partition of $\Omega$ into disjoint triangles $T_k$ with  mesh size $h$ such that no vertex of any triangle lies in the interior of a side of another triangle. Let $X_h$ be the finite-dimensional subspace of $H^{1}_{0}(\Omega)$ consisting of continuous functions on closure $\bar{\Omega}$ of $\Omega$ which are linear in each triangle ${T}_k$ and vanishes on $\partial\Omega$
\begin{equation*}
X_h:=\{v\in C_{0}(\bar{\Omega}): v_{|T_{k}} \ \text{is \ a \ linear \ polynomial} \ \forall \ T_{k}\in \mathcal{T}_{h}\}.
\end{equation*}
Let $\{P_i\}_{i=1}^{M}$ be the interior vertices of $\mathcal{T}_h$ and $\phi_i(x)$ be the pyramid function in $X_h$ which takes the value one at each interior vertex but vanishes at other vertices. Then $\{\phi_i(x)\}_{i=1}^{M}$ forms a basis for the space $X_h.$ For the fully-discrete scheme, we also need to approximate the time-fractional derivative.
Assume that $0=t_0<t_1<...<t_N=T$ be a given partition of time interval [0,T] with step length $\Delta t=\frac{T}{N}$ for some positive integer $N$.
The Gr$\ddot{\mathrm{u}}$nwald-Letnikov approximation to the Riemann-Liouville fractional derivatives is given by

\begin{equation}\label{A9}
^RD^{\alpha}_{t_n}u ={\Delta t}^{-\alpha}\sum_{i=0}^{n} w^{(\alpha)}_{n-i} u(x,t_i)+E_{n}, \quad \mbox{with}\quad\sum_{i=0}^{\infty} w^{(\alpha)}_{i}\xi^i=(1-\xi)^{\alpha}.
\end{equation}
Here the weights $w^{(\alpha)}_i$ are computed as: $w^{(\alpha)}_i=(-1)^i\frac{\Gamma(\alpha+1)}{\Gamma(i+1)\Gamma(\alpha-i+1)}.$\\
The truncation error $E_{n}$ in \eqref{A9}  satisfies the following estimate \cite{Podlubny}
\begin{equation}\label{trunction1}
\|E_n\|\leq C{\Delta}t.
\end{equation}
The approximation given in \eqref{A9}  is $O(\Delta t)$ accurate. In order to obtain $O(\Delta t^2)$ accuracy, authors in \cite{jin2017analysis} derived the following approximation to the Riemann-Liouville fraction derivative $^RD^{\alpha}_{t}$ at point $t= t_{n-\frac{\alpha}{2}}$
\begin{equation}\label{A10}
^RD^{\alpha}_{\Delta t}u(x,t_n)= \ ^RD^{\alpha}_{t_{n-\frac{\alpha}{2}}}u(x,t)+O(\Delta t^2),
\end{equation}
where $^RD^{\alpha}_{\Delta t}$ is the discrete fractional differential operator defined as
\begin{equation}\label{discreteope}
^RD^{\alpha}_{\Delta t}u(x,t_n):={\Delta t}^{-\alpha}\sum_{i=0}^{n} w^{(\alpha)}_{n-i} u(x,t_i).
\end{equation}
\begin{remark}\label{remark}
	The approximation \eqref{A10} was first observed by Dimitrov \cite{dimitrov2013} under the regularity assumption and certain compatibility conditions, i.e., $u \in C^4[0,T]$, $u(0)=0$, $u_t(0)=0$, $u_{tt}(0)=0$. In this work, we establish the theoretical and numerical results based on the regularity assumption and these compatibility conditions for the solution $u$ of the problem \eqref{cuc:1.1}-\eqref{cuc:1.4}. 
\end{remark}
Now, we present fractional Crank-Nicolson-Galerkin finite element scheme to solve the problem \eqref{cuc:1.1}-\eqref{cuc:1.4}. For convenience, set $U_h^{n,\alpha}=(1-\frac{\alpha}{2})U_h^{n}+\frac{\alpha}{2}U_h^{n-1}.$
Then the scheme is to find $U_h^{n}\in X_h$ such that for each $n=1,2,...,N,$ we have
\begin{equation} \label{fully discrete}
\begin{split}
\langle ^RD^{\alpha}_{\Delta t}U_h^{n}, w_{h}\rangle+\langle\nabla U_h^{n,\alpha}, \nabla w_{h}\rangle&=\langle f(U_h^{n,\alpha}), w_{h}\rangle,\quad \forall \  w_{h}\in X_{h},\\
U_h^0=&0.\\
\end{split}
\end{equation}
Clearly for $\alpha=1$ scheme \eqref{fully discrete} recovers the classical Crank-Nicolson scheme. Thus fractional Crank-Nicolson scheme can be seen as an extension of the classical Crank-Nicolson scheme for the fractional order partial differential equations.\\
Again, from the definition of discrete fractional operator $^RD^{\alpha}_{\Delta t},$ we can rewrite equation  (\ref{fully discrete}) as follows
\begin{equation}\label{fully22}
\begin{split}
{\Delta t^{-\alpha}}w_0^{(\alpha)}\langle U_h^{n},w_h\rangle+\langle\nabla U_h^{n,\alpha}, \nabla w_{h}\rangle&=\langle f(U_h^{n,\alpha}), w_{h}\rangle-{\Delta t^{-\alpha}}\sum_{j=1}^{n-1}w_{n-j}^{(\alpha)}\langle U_h^{j},w_h\rangle.
\end{split}
\end{equation}
The discrete formulation \eqref{fully discrete}  (or \eqref{fully22}) gives us a system of nonlinear equations. For solving this system of nonlinear equations, we use Newton's method. First, we recall $\{\phi_i \}_{1\leq i\leq M}$ is the $M$ dimensional basis of $X_h$ associated with nodes of $\mathcal{T}_h.$ For some $\beta_{i}^n$ we can write solution $U_h^n$ of \eqref{fully discrete} (or \eqref{fully22}) as
\begin{equation}\label{values}
U_h^n=\sum_{i=1}^{M}\beta_{i}^n\phi_i.
\end{equation}
Define $\mathbold{\beta^n}$ := $[\beta_1^n, \beta_2^n,...,\beta_M^n]^{\prime}$.
Now using the value of $U_h^n$ from \eqref{values} in \eqref{fully22}, we get the following nonlinear algebraic equation
\begin{equation}\label{newton}
\begin{split}
H_{i}(\mathbold{\beta^n})=& H_{i}(U_h^n)=0,\quad 1\leq i\leq M,
\end{split}
\end{equation}
where
\begin{equation}
\begin{split}
H_{i}(U_h^n)={\Delta t^{-\alpha}}w_0^{(\alpha)}\langle U_h^{n},\phi_i\rangle &+\langle\nabla U_h^{n,\alpha}, \nabla \phi_i\rangle-\langle f(U_h^{n,\alpha}), \phi_i\rangle\\
&+{\Delta t^{-\alpha}}\sum_{j=1}^{n-1}w_{n-j}^{(\alpha)}\langle U_h^{j},\phi_i\rangle.\\
\end{split}
\end{equation}
If we use Newton's method in \eqref{newton}, we get following matrix system for the correction term\\
\begin{equation}
\begin{split}
\mathbold{J} \mathbold{\beta^n} = \mathbold{H},
\end{split}
\end{equation}
where $\mathbold{H}=[H_{1}, H_{2},...,H_{M}]^{\prime},$  and entries of $\mathbold{J}=\mathbold{J}_{(M\times M)},$ is given below
\begin{equation*}
\begin{split}
(\mathbold{J})_{li}=\frac{\partial H_{i}}{\partial \beta_{l}^{n}}(U_h^n, V_h^n)&={\Delta t^{-\alpha}}w_0^{(\alpha)}\langle \phi_l,\phi_i\rangle +\left(1-\frac{\alpha}{2}\right)\langle \nabla \phi_l , \nabla \phi_i\rangle\\
&\hspace{2cm}-\left(1-\frac{\alpha}{2}\right)\Big\langle \frac{\partial f(U_h^{n,\alpha})}{\partial U^{n}_{h}}\phi_l, \phi_i\Big\rangle,\\
\end{split}
\end{equation*}
where  $1\leq i,\ l\leq M.$\\
Next, we prove the existence and uniqueness of the solution for the fully-discrete problem. The following proposition, which is a consequence of the Brouwer fixed point theorem is required to prove the existence and uniqueness of the fully-discrete solution.
\begin{proposition}\cite{sk1,Thomee}\label{Brouwer}
	Let $\mathcal{H}$ be a finite-dimensional Hilbert space with scalar product $\langle\cdot,\cdot\rangle$ and norm $|\cdot|.$ Let $S:\mathcal{H}\rightarrow\mathcal{H}$ be a continuous map such that
	\begin{equation*}
	\langle S(v),v\rangle\textgreater 0 \quad \forall \  v\in \mathcal{H}\quad \mbox{with}\quad |v|=\rho,\quad \rho\textgreater 0.\\
	\end{equation*}
	Then, there exists $w\in\mathcal{H}$ such that
	\begin{equation*}
	S(w)=0 \quad \mbox{and}\quad |w|\textless \rho.
	\end{equation*}
\end{proposition}
\begin{theorem}
	Let $U^0_{h},\ U^1_{h},..., U^{n-1}_{h}$ are given. Then for all $1\leq n\leq N,$ there exists a unique solution $U^n_{h}$ of the problem \eqref{fully discrete}.
\end{theorem}
\begin{proof}
	Rewriting equation \eqref{fully22}  as follows
	\begin{equation}\label{fullyd}
	\begin{split}
	\langle U_h^{n},w_h\rangle+{\Delta t}^{\alpha}\langle \nabla U_h^{n,\alpha}, \nabla w_{h}\rangle& -{\Delta t}^{\alpha} \langle f(U_h^{n,\alpha}), w_{h}\rangle+ \sum_{j=1}^{n-1}w_{n-j}^{(\alpha)}\langle U_h^{j},w_h\rangle=0.
	\end{split}
	\end{equation}
	Multiplying by $(1-\frac{\alpha}{2})$ in \eqref{fullyd} to get
	\begin{equation}\label{fullyd2}
	\begin{split}
	\langle U_h^{n,\alpha},w_h\rangle+{\Delta t}^{\alpha}\langle \nabla U_h^{n,\alpha}, \nabla w_{h}\rangle& -{\Delta t}^{\alpha} \langle f(U_h^{n,\alpha}), w_{h}\rangle\\
	&+ \sum_{j=1}^{n-1}w_{n-j}^{(\alpha)}\langle U_h^{j},w_h\rangle-\frac{\alpha}{2}\langle U_h^{n-1},w_h\rangle=0.
	\end{split}
	\end{equation}
	Equations \eqref{fullyd} and \eqref{fullyd2} are equivalent in the sense that the solution of \eqref{fullyd} is the solution of \eqref{fullyd2} and vice-versa.\\
	Define operator $G:X_h \rightarrow X_h$ such that
	
	\begin{equation}\label{fixedmap}
	\begin{split}
	\langle G(X^{n,\alpha}),W\rangle:=\langle X^{n,\alpha},W\rangle &+{\Delta t}^{\alpha}\langle\nabla X^{n,\alpha},\nabla W\rangle-{\Delta t}^{\alpha}\langle f(X^{n,\alpha}),W\rangle\\
	&+\sum_{j=1}^{n-1}w^{(\alpha)}_{n-j}\langle U^{j}_{h},W\rangle-\frac{\alpha}{2}\langle U_h^{n-1},W\rangle.
	\end{split}
	\end{equation}
	It is clear that $G$ is a continuous map. By choosing $W=X^{n,\alpha}$ in equation \eqref{fixedmap}, we get
	\begin{equation}\label{A5}
	\begin{split}
	\langle G(X^{n,\alpha}),&X^{n,\alpha}\rangle=\langle X^{n,\alpha},X^{n,\alpha}\rangle+\Delta t^{\alpha}\langle\nabla X^{n,\alpha}, \nabla X^{n,\alpha}\rangle\\
	&\hspace{2cm}-\Delta t^{\alpha}\langle f(X^{n,\alpha}),X^{n,\alpha}\rangle
	+\sum_{j=1}^{n-1}w_{n-j}^{(\alpha)}\langle U^{j}_{h},X^{n,\alpha}\rangle-\frac{\alpha}{2}\langle U_h^{n-1},X^{n,\alpha}\rangle.
	\end{split}
	\end{equation}
	From the hypothesis $H,$ we have
	\begin{equation*}
	\|f(X^{n,\alpha})\|\leq L\|X^{n,\alpha}\|+\|f(0)\|.
	\end{equation*}
	This shows that
	\begin{equation}\label{f1bound}
	\|f(X^{n,\alpha})\|\leq a (1+\|X^{n,\alpha}\|), \quad \ a>0.
	\end{equation}
	Using Cauchy-Schwarz inequality in \eqref{A5} together with \eqref{f1bound} and $w_j^{(\alpha)}\textless 0,$ $1\leq j\leq n$, we obtain
	\begin{equation}\label{A6}
	\begin{split}
	\langle G(X^{n,\alpha}),X^{n,\alpha}\rangle&\geq\|X^{n,\alpha}\|^2+\Delta t^{\alpha}\|\nabla X^{n,\alpha}\|^2\\
	&\hspace{1cm}-\Delta t^{\alpha} a(1+\|X^{n,\alpha}\|)\|X^{n,\alpha}\|\\
	&\hspace{1.5cm}+\sum_{j=1}^{n-1}w_{n-j}^{(\alpha)}\|U^{j}_{h}\|\|X^{n,\alpha}\|-\frac{\alpha}{2} \|U_h^{n-1}\|\|X^{n,\alpha}\|.
	\end{split}
	\end{equation}
	Since $\Delta t^{\alpha}\|\nabla X^n\|\textgreater 0,$ \eqref{A6} can be written as
	\begin{equation*}
	\begin{split}
	\langle G(X^{n,\alpha}),X^{n,\alpha}\rangle &\geq\Big((1-\Delta t^{\alpha} a)\|X^{n,\alpha}\|-\Delta t^{\alpha} a\\
	&\hspace{1cm}+\sum_{j=1}^{n-1}w_{n-j}^{(\alpha)}\|U^{j}_{h}\|-\frac{\alpha}{2}\|U_h^{n-1}\|\Big)\|X^{n,\alpha}\|.
	\end{split}
	\end{equation*}
	Then $\langle G(X^{n,\alpha}),X^{n,\alpha}\rangle>0,$ if the following result holds
	\begin{equation*}
	\begin{split}
	(1-\Delta t^{\alpha} a)\|X^{n,\alpha}\|-\Delta t^{\alpha} a+\sum_{j=1}^{n-1}w_{n-j}^{(\alpha)}\|U^{j}_{h}\|-\frac{\alpha}{2}\|U_h^{n-1}\|&>0.
	\end{split}
	\end{equation*}
	Choosing $\Delta t^{\alpha}<\frac{1}{a},$ then $\exists$ $X^{n,\alpha}$ such that
	\begin{equation*}
	\begin{split}
	\|X^{n,\alpha}\|&>\frac{1}{(1-\Delta t^{\alpha} a)}\Big(\Delta t^{\alpha} a
	-\sum_{j=1}^{n-1}w^{(\alpha)}_{n-j}\|U^{j}_{h}\|+\frac{\alpha}{2}\|U_h^{n-1}\|\Big),
	\end{split}
	\end{equation*}
	which implies  $\langle G(X^{n,\alpha}),X^{n,\alpha}\rangle>0.$ Thus for
	$\|X^{n,\alpha}\|=\rho,$ we have\\
	\begin{equation*}
	\begin{split}
	\langle G(X^{n,\alpha}),X^{n,\alpha}\rangle>0.
	\end{split}
	\end{equation*}
	Hence, the existence of discrete solution is assured from Proposition \ref{Brouwer}.
	\par Next, we prove the uniqueness of the solution $U^{n,\alpha}_{h}$ for the problem \eqref{fullyd2}. Assume that $U^{n,\alpha}_{h1}$ and $U^{n,\alpha}_{h2}$ are two solutions of the problem \eqref{fullyd2}. For simplicity, we denote $U_1=U^{n,\alpha}_{h1},$ $U_2=U^{n,\alpha}_{h2}.$ Then from \eqref{fullyd2}, we obtain
	\begin{equation}\label{unique}
	\begin{split}
	\langle U_1-U_2,w_h\rangle+\Delta t^{\alpha}\langle \nabla (U_1-U_2), \nabla w_{h}\rangle &=\Delta t^{\alpha}\langle f(U_1)-f(U_2), w_{h}\rangle.
	\end{split}
	\end{equation}
	Setting $w_h=U_1-U_2=r$ in \eqref{unique} and using hypothesis $H,$ we obtain
	\begin{equation}\label{dd1}
	\|r\|^2\leq \Delta t^{\alpha}L(\|r\|)\|r\|.\\
	\end{equation}
	Taking $\Delta t^{\alpha}<\frac{1}{L}$ sufficiently small, we get
	\begin{equation*}
	\|r\|^2 \leq 0.\\
	\end{equation*}
	This completes the proof.
\end{proof}
\section{{\it{\textbf{A priori}}} bound and error analysis}
Here, we provide {\it{a priori}} bound and {\it{a priori}} error estimate for the fully-discrete scheme \eqref{fully discrete}. For the derivation of these estimates, we require certain results (e.g., fractional Gronwall type inequality) which are discussed below.\\
First, we recall the Gr$\ddot{\mathrm{u}}$nwald weights $w^{(\alpha)}_i$ and these weights can be recursively computed via $w^{(\alpha)}_0=1,$ and
\begin{equation*}
w^{(\alpha)}_i=\left(1-\frac{\alpha+1}{i}\right)w^{(\alpha)}_{i-1}, \quad \mbox{for} \ i\geq 1.\\
\end{equation*}
Define $g_{n}^{(\alpha)}:=\sum_{i=0}^{n} w^{(\alpha)}_i$, then $g^{(\alpha)}_0=w^{(\alpha)}_0$ and $w^{(\alpha)}_i=g^{(\alpha)}_i - g^{(\alpha)}_{i-1}$
for $1\leq i\leq n.$ Since weights $w^{(\alpha)}_i$ possess following properties
\begin{equation}
w^{(\alpha)}_0=1, \quad -1<w^{(\alpha)}_1<w^{(\alpha)}_2<...<w^{(\alpha)}_i<...<0,\quad \sum_{i=0}^{\infty}w^{(\alpha)}_i=0, 
\end{equation}
therefore $g^{(\alpha)}_{i-1} > g^{(\alpha)}_{i}$ for $i\geq 1$.
By using the definition of $g^{(\alpha)}_n,$  equation \eqref{discreteope}  can be re-written  as  
\begin{equation}\label{A7}
^RD^{\alpha}_{\Delta t}u(x,t_n)={\Delta t}^{-\alpha}\sum_{i=1}^{n} (g_i^{(\alpha)}-g_{i-1}^{(\alpha)})u(x,t_{n-i})+\Delta t^{-\alpha} g_0^{(\alpha)}u(x,t_{n}).
\end{equation}
Using the fact $u(x,t_0)=0,$ from \eqref{A7} we get
\begin{equation}\label{weight}
^RD^{\alpha}_{\Delta t}u(x,t_n)={\Delta t}^{-\alpha}\sum_{i=1}^{n} g_{n-i}^{(\alpha)}\delta u(x,t_{i}),
\end{equation}
where $\delta u(x,t_{i})=u(x,t_{i})-u(x,t_{i-1})$ $\forall \ i=1,...,n.$\\
With the help of above notations, we now prove some results which are useful in the derivation of fractional Gronwall type inequality and error analysis. For proving these results, we basically borrow the idea given in \cite{D.}. 
\begin{lemma}\label{dkl}
	Consider the sequence $\{\phi_n\}$ given by 
	\begin{equation}
	\phi_0=1, \quad \phi_n=\sum_{i=1}^{n}(g^{(\alpha)}_{i-1}-g^{(\alpha)}_i) \ \phi_{n-i},\quad n\geq 1.
	\end{equation}
	Then $\{\phi_n\}$ satisfies the following properties
	\begin{equation}\label{d1}
	(i) \  0<\phi_n<1, \quad \sum_{i=j}^{n}\phi_{n-i} \  g^{(\alpha)}_{i-j}=1, \quad 1\leq j\leq n, \hspace{3cm}
	\end{equation}
	\begin{equation}\label{d2}
	(ii) \  \frac{1}{\Gamma (\alpha)} \sum_{i=1}^{n}\phi_{n-i} \leq \frac{n^{\alpha}}{\Gamma (1+\alpha)}, \hspace{6.4cm}
	\end{equation}
	\begin{equation}\label{d3}
	(iii) \ \frac{1}{\Gamma (\alpha) \Gamma (1+(k-1)\alpha)} \sum_{i=1}^{n-1}\phi_{n-i}\ i^{(k-1)\alpha} \leq \frac{n^{k \alpha}}{\Gamma (1+ k \alpha)},  
	\hspace{2cm}
	\end{equation} 
	where $k=1,2,...$.
\end{lemma}
\begin{proof}
	The proof of the first estimate \eqref{d1} is similar to part (i) of Lemma 3.2 in \cite{D.}. In order to prove \eqref{d2}, we first prove
	\begin{equation}\label{d4}
	\phi_n=(-1)^n \binom{-\alpha}{n}\quad n\geq 0,
	\end{equation}
	where $(-1)^n \binom{-\alpha}{n}$ represents the binomial coefficients of the expansion $(1-\xi)^{-\alpha}=\sum_{n=0}^{\infty}\phi_n \xi^n.$ The fractional binomial coefficients $\binom{-\alpha}{n}$ are defined in a similar way as the integer binomial coefficients with the gamma function, i.e.,
	\begin{equation*}
	\binom{-\alpha}{n}=\frac{\Gamma(-\alpha+1)}{\Gamma (n+1)\Gamma (-\alpha-n+1)}=\frac{-\alpha(-\alpha-1)(-\alpha-2)...(-\alpha-n+1)}{\Gamma (n+1)}.
	\end{equation*}
	For proving \eqref{d4}, we use induction hypothesis on $n.$ It is obvious that \eqref{d4} holds for $n=0$. Next we assume that \eqref{d4} holds for $n\leq m$. Then, we have 
	\begin{equation*}
	\phi_{m+1}=-\sum_{i=1}^{m+1}w^{(\alpha)}_i\phi_{m+1-i}=\sum_{i=1}^{m+1} (-1)^{i+1}\binom{\alpha}{i}(-1)^{m+1-i}\binom{-\alpha}{m+1-i}
	\end{equation*}
	\begin{equation}\label{A8}
	=(-1)^{m}\sum_{i=0}^{m+1} \binom{\alpha}{i}\binom{-\alpha}{m+1-i}+(-1)^{m+1}\binom{\alpha}{0}\binom{-\alpha}{m+1}.
	\end{equation}
	By using the well-known Vandermonde convolution \cite{sriv} in \eqref{A8}, we get 
	\begin{equation*}
	\phi_{m+1}= (-1)^{m}\binom{0}{m+1}+(-1)^{m+1}\binom{-\alpha}{m+1}=(-1)^{m+1}\binom{-\alpha}{m+1}.
	\end{equation*}
	This completes the proof of \eqref{d4}.\\
	From \eqref{d4} one can obtain $\phi_0=1,$ $\phi_n=\prod_{l=1}^{n}(1+\frac{\alpha-1}{l}).$ By the application of the inequality $\ln(1+x)\leq x$ for $x>-1,$ we have 
	\begin{equation*}
	\ln \phi_n=\sum_{l=1}^{n}\ln\Big(1+\frac{\alpha-1}{l}\Big)\leq (\alpha-1)\sum_{l=1}^{n}\frac{1}{l}\leq (\alpha-1) \int_{1}^{n+1}\frac{1}{x} dx.
	\end{equation*}
	This implies that, $\phi_n\leq (n+1)^{\alpha-1}$ for $n\geq 0.$ \\
	Now, we prove \eqref{d2}, take 
	\begin{equation*}
	\frac{1}{\Gamma (\alpha)} \sum_{i=1}^{n}\phi_{n-i} =\frac{1}{\Gamma (\alpha)} \sum_{i=0}^{n-1}\phi_{i} \leq\frac{1}{\Gamma (\alpha)} \sum_{i=1}^{n}\frac{1}{i^{1-\alpha}}\leq \frac{1}{\Gamma (\alpha)}  \int_{0}^{n}\frac{1}{x^{1-\alpha}}dx= \frac{n^{\alpha}}{\Gamma (1+\alpha)}.
	\end{equation*}
	This completes the proof of \eqref{d2}.\\
	Further, we prove \eqref{d3}. It is obvious that \eqref{d3} holds for $n=1,$ so to prove \eqref{d3} we use the inequality  $\phi_n\leq (n+1)^{\alpha-1}$ for $n\geq 0$  which implies $\phi_n\leq n^{\alpha-1}$ for $n\geq 1,$ then
	\begin{equation}\label{d5}
	\begin{split}
	\sum_{i=1}^{n-1}\phi_{n-i} \ i^{(k-1)\alpha} &\leq   \sum_{i=1}^{n-1}  (n-i)^{\alpha-1} \ i^{(k-1)\alpha}\\
	&= \sum_{i=1}^{n-1}i^{\alpha-1} (n-i)^{(k-1)\alpha},\quad \mbox{for} \quad n\geq 2.
	\end{split}
	\end{equation}
	For any fixed $n\geq 2$ and $k\geq 1,$ we define 
	\begin{equation*}
	f(x):=x^{\alpha-1}(n-x)^{(k-1)\alpha}, \quad x\in [1, n-1].
	\end{equation*} 
	Therefore, $f'(x)<0$ $\forall  \ x\in [1,n-1],$ $n\geq 2$ and $k\geq 1.$ From this we can conclude that $f(x)$ is decreasing function in $[1,n-1]$. Thus, we can write \eqref{d5} as follows
	\begin{equation}\label{d6}
	\begin{split}
	\sum_{i=1}^{n-1}i^{\alpha-1} (n-i)^{(k-1)\alpha} &\leq   \int_{0}^{n-1}x^{\alpha-1} (n-x)^{(k-1)\alpha}dx\\
	&\leq  \int_{0}^{n}x^{\alpha-1} (n-x)^{(k-1)\alpha}dx,
	\end{split}
	\end{equation} 
	where we have used the fact that $\int_{n-1}^{n}f(x)dx\geq 0$ for $n\geq 2.$ Now, we have
	\begin{equation}\label{d7}
	\begin{split}
	\int_{0}^{n}x^{\alpha-1} (n-x)^{(k-1)\alpha}dx&=n^{k\alpha}\int_{0}^{1}y^{\alpha-1}(1-y)^{(k-1)\alpha}dy\\
	&=n^{k\alpha} B(\alpha,1+(k-1)\alpha),
	\end{split}
	\end{equation}
	where  $B(m,n)=\frac{\Gamma (m)\Gamma (n)}{\Gamma (m+n)}, m>0,\ n>0$ denotes the beta function. Finally, using \eqref{d7} and \eqref{d6} in \eqref{d5} to get
	\begin{equation*}
	\begin{split}
	\frac{1}{\Gamma (\alpha) \Gamma (1+(k-1)\alpha)} \sum_{i=1}^{n-1}\phi_{n-i} \ i^{(k-1)\alpha} &\leq \frac{n^{k\alpha} B(\alpha,1+(k-1)\alpha)}{\Gamma (\alpha) \Gamma (1+(k-1)\alpha)}  \\
	&= \frac{n^{k\alpha}}{\Gamma (1+k\alpha)}.
	\end{split}
	\end{equation*} 
	This completes the proof of Lemma \eqref{dkl}. 
\end{proof}
\begin{lemma}\label{dkl2}
	Consider the matrix
	\begin{equation}
	W = 2\mu\Delta t^{\alpha}
	\left [
	\begin{array}{ccccc}
	0&\phi_1&\cdots &\phi_{n-2}  &\phi_{n-1} \\
	0&0&\cdots &\phi_{n-3}  &\phi_{n-2} \\
	\vdots & \vdots & \ddots & \vdots & \vdots\\
	0&0&\cdots &0  &\phi_1 \\
	0&0&\cdots &0  &0 \\
	\end{array}
	\right ]_{n\times n}.
	\end{equation}
	Then, $W$ satisfies the following properties
	\begin{equation*}
	(i) \ W^l=0, \quad l\geq  n, \hspace{12cm}
	\end{equation*}
	\begin{equation*}
	(ii)\ W^k \mathbold{e}\leq \frac{1}{\Gamma(1+k\alpha)}[(2\Gamma (\alpha)\mu t_n^{\alpha})^k,(2\Gamma (\alpha)\mu t_{n-1}^{\alpha})^k,...,(2\Gamma (\alpha)\mu t_1^{\alpha})^k]^{\prime}, \hspace{.6cm} {k=0,1,2,...,} \hspace{3cm}
	\end{equation*}
	\begin{equation*}
	(iii) \ \sum_{k=0}^{l}W^k \mathbold{e}=\sum_{k=0}^{n-1}W^k \mathbold{e}\leq [E_{\alpha}(2\Gamma (\alpha)\mu t_n^{\alpha}),E_{\alpha}(2\Gamma (\alpha)\mu t_{n-1}^{\alpha}),...,E_{\alpha}(2\Gamma (\alpha)\mu t_1^{\alpha})]^{\prime}, \hspace{.5cm} l\geq n, \hspace{2cm}
	\end{equation*} 
	where $\mathbold{e}=[1,1,..,1]^{\prime}\in\mathbb{R}^n$.
\end{lemma}
\begin{proof}
	The proof of Lemma \eqref{dkl2} is similar to Lemma 3.3 of \cite{D.}. 
\end{proof} 
\begin{lemma}\label{mainresult}
	Let $\{a^n,\ b^n| \ n\geq 0\}$ be nonnegative sequences and $\mu_1$, $\mu_2$ be nonnegative constants. For $a^0=0$ and 
	\begin{equation}\label{dl8}
	^RD^{\alpha}_{\Delta t}a^n\leq\mu_1 a^n + \mu_2 a^{n-1}+b^n,\quad n\geq 1,
	\end{equation}
	there exists a positive constant $\Delta t^*$ such that, when $\Delta t\leq \Delta t^*,$
	\begin{equation*}
	a^n\leq 2\Big(\frac{t_n^{\alpha}}{\alpha} \ \max_{0\leq i\leq n}b^i\Big)E_{\alpha}(2 \Gamma (\alpha) \mu t_n^{\alpha}), \quad 1\leq n\leq N,
	\end{equation*}
	where $E_{\alpha}(z)$=$\sum_{j=0}^{\infty}\frac{z^j}{\Gamma(1+j\alpha) }$ is the Mittag-Leffler function and $\mu=\mu_1+\frac{\mu_2}{\alpha}.$
\end{lemma}
\begin{proof}  From the definition of discrete fractional differential operator \eqref{weight}, we can write \eqref{dl8} as follows  
	\begin{equation}\
	\sum_{j=1}^{i}g^{(\alpha)}_{i-j}\delta a^j\leq\Delta t^{\alpha}(\mu_1 a^j + \mu_2 a^{j-1}+\Delta t^{\alpha}b^j).
	\end{equation}
	Further, by using  the Lemmas \ref{dkl} and \ref{dkl2}, the proof of Lemma \ref{mainresult}  follows in similar line as proof of Lemma 3.1 in \cite{D.}.\\
\end{proof}
\begin{lemma}\label{inequality}
	For any sequence $\{e^k\}_{k=0}^{N}\subset X_h$, following inequality holds
	\begin{equation}
	\langle ^RD^{\alpha}_{\Delta t}e^{k},(1-\frac{\alpha}{2})e^k+\frac{\alpha}{2}e^{k-1}\rangle\geq \frac{1}{2} \ ^RD^{\alpha}_{\Delta t}\|e^k\|^2,\quad \mbox{for} \quad 1\leq k\leq N.
	\end{equation}
\end{lemma}
\begin{proof}
	We have
	\begin{equation*}
	\begin{split}
	\langle ^RD^{\alpha}_{\Delta t}e^{k},(1-\frac{\alpha}{2})e^k+\frac{\alpha}{2}e^{k-1}\rangle&=(1-\frac{\alpha}{2})\langle ^RD^{\alpha}_{\Delta t}e^{k},e^k\rangle+\frac{\alpha}{2}\langle ^RD^{\alpha}_{\Delta t}e^{k},e^{k-1}\rangle\\
	&=\Delta t^{-\alpha}\Big((1-\frac{\alpha}{2})\sum_{j=0}^{k}w_{k-j}^{(\alpha)}\langle e^j,e^k\rangle+\frac{\alpha}{2}\sum_{j=0}^{k}w_{k-j}^{(\alpha)}\langle e^j,e^{k-1}\rangle\Big)\\
	\end{split}
	\end{equation*}
	\begin{equation*}
	\begin{split}
	\hspace{3.5cm}	&=\Delta t^{-\alpha}\Big((1-\frac{\alpha}{2})w_0^{(\alpha)}\|e^k\|^2+\frac{\alpha}{2}w_1^{(\alpha)}\|e^{k-1}\|^2\\
	&\hspace{2cm} +\big((1-\frac{\alpha}{2})w_1^{(\alpha)}+\frac{\alpha}{2}w_0^{(\alpha)}\big)\langle e^k,e^{k-1}\rangle\\
	\end{split}
	\end{equation*}
	\begin{equation*}
	\hspace{8cm}+(1-\frac{\alpha}{2})\sum_{j=0}^{k-2}w_{k-j}^{(\alpha)}\langle e^j,e^k\rangle+\frac{\alpha}{2}\sum_{j=0}^{k-2}w_{k-j}^{(\alpha)}\langle e^j,e^{k-1}\rangle\Big)\\
	\end{equation*}
	\begin{equation*}
	\hspace{3cm}	\geq \Delta t^{-\alpha}\Big((1-\frac{\alpha}{2})w_0^{(\alpha)}\|e^k\|^2+\frac{\alpha}{2}w_1^{(\alpha)}\|e^{k-1}\|^2\\
	\end{equation*}
	\begin{equation*}
	\begin{split}
	\hspace{4cm}&\quad \quad+\big((1-\frac{\alpha}{2})w_1^{(\alpha)}+\frac{\alpha}{2}w_0^{(\alpha)}\big)\frac{\|e^k\|^2+\|e^{k-1}\|^2}{2}\\
	&\hspace{2cm}+(1-\frac{\alpha}{2})\sum_{j=0}^{k-2}w_{k-j}^{(\alpha)}\frac{\|e^j\|^2+\|e^k\|^2}{2}\\
	&\hspace{3cm}+\frac{\alpha}{2}\sum_{j=0}^{k-2}w_{k-j}^{(\alpha)}\frac{\|e^j\|^2+\|e^{k-1}\|^2}{2}\Big),
	\end{split}
	\end{equation*}
	where we have used the fact $(1-\frac{\alpha}{2})w_1^{(\alpha)}+\frac{\alpha}{2}w_0^{(\alpha)}<0$ and $w_j^{(\alpha)}<0 \ \forall j\geq 1.$\\
	Next, we have
	\begin{equation*}
	\begin{split}
	\langle ^RD^{\alpha}_{\Delta t}e^{k},(1-\frac{\alpha}{2})e^k+\frac{\alpha}{2}e^{k-1}\rangle
	&\geq \Delta t^{-\alpha}\Big(\big((1-\frac{\alpha}{2})w_0^{(\alpha)}+\frac{1}{2}(1-\frac{\alpha}{2})w_1^{(\alpha)}+\frac{\alpha}{4}w_0^{(\alpha)}\big)\|e^k\|^2\\
	\end{split}
	\end{equation*}
	\begin{equation*}
	\begin{split}
	\hspace{6cm}	&+\big(\frac{\alpha}{2}w_1^{(\alpha)}+\frac{1}{2}(1-\frac{\alpha}{2})w_1^{(\alpha)}+\frac{\alpha}{4}w_0^{(\alpha)}\big)\|e^{k-1}\|^2\\
	&\hspace{1cm}+\frac{1}{2}\sum_{j=0}^{k-2}w_{k-j}^{(\alpha)}\|e^j\|^2+\frac{1}{2}(1-\frac{\alpha}{2})\sum_{j=0}^{k-2}w_{k-j}^{(\alpha)}\|e^{k}\|^2\\
	&\hspace{5cm}+\frac{\alpha}{4}\sum_{j=0}^{k-2}w_{k-j}^{(\alpha)}\|e^{k-1}\|^2\Big)
	\end{split}
	\end{equation*}
	\begin{equation*}
	\begin{split}
	\hspace{5cm}&= \Delta t^{-\alpha}\Big(\big((1-\frac{\alpha}{2})w_0^{(\alpha)}+\frac{\alpha}{4}w_0^{(\alpha)}-\frac{1}{2}(1-\frac{\alpha}{2})w_0^{(\alpha)}\big)\|e^k\|^2\\
	&\hspace{1.5cm}+\big(\frac{\alpha}{2}w_1^{(\alpha)}+\frac{1}{2}(1-\frac{\alpha}{2})w_1^{(\alpha)}-\frac{\alpha}{4}w_1^{(\alpha)}\big)\|e^{k-1}\|^2\\
	\end{split}
	\end{equation*}
	\begin{equation*}
	\begin{split}
	\hspace{7cm}&+\frac{1}{2}\sum_{j=0}^{k-2}w_{k-j}^{(\alpha)}\|e^j\|^2+\frac{1}{2}(1-\frac{\alpha}{2})\sum_{j=0}^{k}w_{k-j}^{(\alpha)}\|e^{k}\|^2\\
	&\hspace{3cm}+\frac{\alpha}{4}\sum_{j=0}^{k}w_{k-j}^{(\alpha)}\|e^{k-1}\|^2\Big)\\
	\end{split}
	\end{equation*}
	\begin{equation*}
	\hspace{3cm}	\geq \frac{\Delta t^{-\alpha}}{2}\sum_{j=0}^{k}w_{k-j}^{(\alpha)}\|e^j\|^2=\frac{1}{2} \ ^RD^{\alpha}_{\Delta t}\|e^k\|^2.
	\end{equation*}
	This completes the proof of Lemma \ref{inequality}.
\end{proof}
\noindent In the following theorem we provide {\it{a priori}} bound for the fully-discrete solution $U^n_{h}$.
\begin{theorem}\label{stability1 theorem}
	Let $U_h^n$ be the solution of the fully-discrete scheme $(\ref{fully discrete})$.  Then, there exists positive constant  ${\Delta t}^{*}$ such that when $\Delta t\leq {\Delta t}^{*},$ the solution $U_h^n$ satisfies
	\begin{equation}
	\|U_h^n\|\leq C,
	\end{equation}
	where  $n= 1,2,...,N$ and $C$ is a positive constant independent of $h$ and $\Delta t.$
\end{theorem}
\begin{proof}
	From  (\ref{fully discrete}), we have
	\begin{equation}\label{A3}
	\langle ^RD^{\alpha}_{\Delta t}U_h^{n}, w_{h}\rangle+\langle\nabla U_h^{n,\alpha}, \nabla w_{h}\rangle=\langle f(U_h^{n,\alpha}), w_{h}\rangle,\quad \forall \  w_{h}\in X_{h}.\\
	\end{equation}
	Setting $w_{h}=U_h^{n,\alpha}$ in  \eqref{A3} to obtain
	\begin{equation} \label{stb1}
	\langle ^RD^{\alpha}_{\Delta t}U_h^{n},U_h^{n,\alpha}\rangle+\|\nabla U_h^{n,\alpha}\|^2\leq\frac{1}{2}\big(\|f(U_h^{n,\alpha})\|^2+\| U_h^{n,\alpha}\|^2\big).
	\end{equation}
	Using \eqref{f1bound} in \eqref{stb1} to get
	\begin{equation}\label{dp}
	\langle ^RD^{\alpha}_{\Delta t}U_h^{n},U_h^{n,\alpha}\rangle+\|\nabla U_h^{n,\alpha}\|^2
	\leq C\big((1+\| U_h^{n,\alpha}\|)^2+\| U_h^{n,\alpha}\|^2\big).\\
	\end{equation}
	For $a,b\geq 0,$ using the fact $(a+b)^2\leq 2(a^2+b^2)$ in \eqref{dp} we have
	\begin{equation}\label{dp1}
	\langle ^RD^{\alpha}_{\Delta t}U_h^{n},U_h^{n,\alpha}\rangle \leq C\big(1+\| U_h^{n,\alpha}\|^2\big).\\
	\end{equation}
	Using Lemma \ref{inequality} in \eqref{dp1}, we have
	\begin{equation} \label{stb2}
	{^RD}^{\alpha}_{\Delta t}\|U_h^{n}\|^2\leq C\big(1+\| U_h^{n,\alpha}\|^2).\\
	\end{equation}
	From \eqref{stb2} we get
	\begin{equation}\label{A4}
	\begin{split}
	{^RD}^{\alpha}_{\Delta t}\|U_h^{n}\|^2\leq 
	C\Big(1+&\big(1-\frac{\alpha}{2}\big)^2\| U_h^{n}\|^2\\
	&+\big(\frac{\alpha}{2}\big)^2\| U_h^{n-1}\|^2\Big).
	\end{split}
	\end{equation}
	Using Lemma \ref{mainresult} in \eqref{A4}, one can find a positive constant ${\Delta t}^*$ such that when $\Delta t\leq {\Delta t}^*$, then
	\begin{equation*} \label{stb4}
	\begin{split}
	\|U_h^{n}\|^2&\leq C.\\
	\end{split}
	\end{equation*}
	Thus
	\begin{equation*} 
	\begin{split}
	\|U_h^{n}\|&\leq C.\\
	\end{split}
	\end{equation*}
	This completes the proof of Theorem \ref{stability1 theorem} .
\end{proof}
Further, we consider the {\it{a priori}} error estimate for the fully-discrete scheme \eqref{fully discrete}. The direct comparison between $u(x,t_n)=u^n(x)\ \mbox{and}\ U_h^n $ may not yield optimal convergence. Therefore we need to define Ritz projection $R_h:H^1_0(\Omega)\cap H^2(\Omega)\rightarrow X_h$ which satisfies following equation \cite{Thomee}
\begin{equation}
\label{proj}
\langle \nabla R_h w,\nabla v_h\rangle=\langle\nabla w,\nabla v_h\rangle,\hspace{.5cm} \forall \ w\in H_0^1(\Omega)\cap H^2(\Omega),\ v_h\in X_h.\\
\end{equation}
The operator $R_h$ defined in \eqref{proj} satisfies the following approximation property which is useful in the derivation of {\it{a priori}} error estimate.
\begin{theorem}\cite{R.}\label{projerr} There exists a positive constant C, independent of $h$ such that
	\begin{equation}
	\|w-R_hw\|_j\leq Ch^{i-j}\|w\|_i \  , \hspace{.5cm}\forall \  w\in H^{i}\cap H_0^1, \hspace{.2cm} j=0 ,\ 1; \ i=1, \ 2.
	\end{equation}
\end{theorem}
\noindent By using an intermediate projection $R_h,$ we can write the error as follows
\begin{equation*}
\begin{split}
u^n(x)-U_h^n=u^n-U_h^n=&(u^n -R_hu^n)+(R_hu^n-U_h^n)=\rho_h^n+\theta_h^n.\\
\end{split}
\end{equation*}
In the following theorem we derive {\it{a priori}} error estimate for the fully-discrete problem \eqref{fully discrete}.
\begin{theorem}\label{fully1 theorem}
	Let $u^n$ be the solution of $(\ref{cuc:1.1})$- $(\ref{cuc:1.4})$ and $U_h^n$ be the solution of the fully-discrete scheme $(\ref{fully discrete}).$ Then, there exists positive constant  ${\Delta t}^{*}$ such that if $\Delta t\leq {\Delta t}^{*},$ then the following estimate holds
	\begin{equation}
	\|u^n-U_h^n\|\leq C\big(\Delta t^{2}+h^2\big),
	\end{equation}
	where  $n= 1,2,...,N$ and $C$ is a positive constant independent of $h$ and $\Delta t.$
\end{theorem}
\begin{proof}
	For any $w_h\in X_h,$ we have the following estimate for ${\theta_h^n}$
	\begin{equation}\label{dl1}
	\begin{split}
	\langle ^RD^{\alpha}_{\Delta t}{\theta_h^n},w_h\rangle+\langle \nabla{\theta_h^{n,\alpha},\nabla w_h}\rangle=&\langle ^RD^{\alpha}_{\Delta t}(R_hu^n-U_h^n),w_h\rangle+\langle\nabla(R_hu^{n,\alpha}-U_h^{n,\alpha}),\nabla w_h\rangle\\
	=&\langle ^RD^{\alpha}_{\Delta t}{R_hu^n},w_h\rangle+\langle\nabla{R_hu^{n,\alpha}},\nabla w_h\rangle\\
	&-\langle ^RD^{\alpha}_{\Delta t}{U_h^n},w_h\rangle-\langle\nabla{U_h^{n,\alpha}},\nabla w_h\rangle.\\
	\end{split}
	\end{equation}
	Using $(\ref{fully discrete})$ and $(\ref{proj})$ in  \eqref{dl1} to get
	\begin{equation}\label{A1}
	\begin{split}
	\langle ^RD^{\alpha}_{\Delta t}{\theta_h^n},w_h\rangle+\langle\nabla{\theta_h^{n,\alpha},\nabla w_h}\rangle&=\langle ^RD^{\alpha}_{\Delta t}{R_hu^n},w_h\rangle+\langle\nabla{u^{n,\alpha}},\nabla w_h\rangle\\
	-\langle f(u^{n-\frac{\alpha}{2}}),w_h\rangle
	+&\langle f(u^{n-\frac{\alpha}{2}}),w_h\rangle
	-\langle f(U_h^{n,\alpha}), w_{h}\rangle.
	\end{split}
	\end{equation}
	From the weak formulation of \eqref{cuc:1.1} one can get
	\begin{equation}\label{weak form}
	\begin{split}
	\langle ^RD^{\alpha}_{t_{n-\frac{\alpha}{2}}}u, w_h\rangle+\langle\nabla u^{n-\frac{\alpha}{2}}, \nabla w_h\rangle&=\langle f(u^{n-\frac{\alpha}{2}}), w_h\rangle.
	\end{split}
	\end{equation}
	Using \eqref{weak form} in \eqref{A1} to obtain
	\begin{equation}\label{solving}
	\begin{split}
	\langle ^RD^{\alpha}_{\Delta t}{\theta_{h}^n},w_h\rangle+\langle\nabla{\theta_{h}^{n,\alpha},\nabla w_h}\rangle=&\langle ^RD^{\alpha}_{\Delta t}{R_hu^n}-\  ^R{D}^{\alpha}_{t_{n-\frac{\alpha}{2}}}u,w_h\rangle\\
	&+\langle \nabla(u^{n,\alpha}-u^{n-\frac{\alpha}{2}}) ,\nabla w_h\rangle\\
	&+\langle f(u^{n-\frac{\alpha}{2}})-f(U_h^{n,\alpha}), w_{h}\rangle.
	\end{split}
	\end{equation}
	Setting $w_h=\theta_{h}^{n,\alpha}$ in $(\ref{solving})$ and using Cauchy-Schwarz inequality, we have
	\begin{equation*}
	\begin{split}
	\langle ^RD^{\alpha}_{\Delta t}{\theta_{h}^n},\theta_{h}^{n,\alpha}\rangle+\|\nabla{\theta_{h}^{n,\alpha}}\|^2
	\leq & \|^RD^{\alpha}_{\Delta t}{R_hu^n}-{^R{D}}^{\alpha}_{t_{{n-\frac{\alpha}{2}}}}u\|\ \|\theta_{h}^{n,\alpha}\|\\
	&\hspace{1cm}+\|\nabla(u^{n,\alpha}-u^{n-\frac{\alpha}{2}})\|\|\nabla\theta_{h}^{n,\alpha}\|
	\\&\hspace{1.5cm}+\|f(u^{n-\frac{\alpha}{2}})-f(U_h^{n,\alpha})\|\|\theta_{h}^{n,\alpha}\|\\
	\leq & \frac{L}{2}\|u^{n-\frac{\alpha}{2}}-U_h^{n,\alpha}\|^2+\frac{L}{2}\|\theta_{h}^{n,\alpha}\|^2\\
	&\hspace{1cm}+\frac{1}{2}\|\theta_{h}^{n,\alpha}\|^2
	+\frac{1}{2}\|^RD^{\alpha}_{\Delta t}{R_hu^n}-{^R{D}}^{\alpha}_{t_{n-\frac{\alpha}{2}}}u\|^2\\
	&\hspace{3cm}+\frac{1}{2}\|\nabla\theta_{h}^{n,\alpha}\|^2
	+\frac{1}{2}\|\nabla(u^{n,\alpha}-u^{n-\frac{\alpha}{2}})\|^2\\
	\end{split}
	\end{equation*}
	\begin{equation}\label{dl2}
	\begin{split}
	\hspace{5cm}\leq  \frac{L}{2}&\|u^{n-\frac{\alpha}{2}}-U_h^{n,\alpha}\|^2\\
	&+\left(\frac{L+1}{2}\right)\|\theta_{h}^{n,\alpha}\|^2+\frac{1}{2}\|^RD^{\alpha}_{\Delta t}{R_hu^n}-{^R{D}}^{\alpha}_{t_{n-\frac{\alpha}{2}}}u\|^2\\
	&\hspace{1cm}+\frac{1}{2}\|\nabla(u^{n,\alpha}-u^{n-\frac{\alpha}{2}})\|^2
	+\frac{1}{2}\|\nabla\theta_{h}^{n,\alpha}\|^2.
	\end{split}
	\end{equation}
	Note that
	\begin{equation}\label{erro1}
	\begin{split}
	\|u^{n-\frac{\alpha}{2}}-U_h^{n,\alpha}\|&\leq \|u^{n-\frac{\alpha}{2}}-u^{n,\alpha}\|+\|\rho_{h}^{n,\alpha}\|+\|\theta_{h}^{n,\alpha}\|\leq\|\theta_{h}^{n,\alpha}\|+C(\Delta t^2+h^2).
	\end{split}
	\end{equation}
	Also,
	\begin{equation}\label{dl3}
	\begin{split}
	\|^RD^{\alpha}_{\Delta t}{R_hu^n}-{^R{D}}^{\alpha}_{t_{n-\frac{\alpha}{2}}}u\|&\leq  \|^RD^{\alpha}_{\Delta t}{R_hu^n}-{^R{D}}^{\alpha}_{t_{n-\frac{\alpha}{2}}}R_hu\|\\
	&\hspace{3cm}+\|^RD^{\alpha}_{t_{n-\frac{\alpha}{2}}}{R_hu}-{^R{D}}^{\alpha}_{t_{n-\frac{\alpha}{2}}}u\|\\
	&\leq C(\Delta t^{2}+h^2),\\
	\end{split}
	\end{equation}
	and
	\begin{equation}\label{erro2}
	\begin{split}
	\|\nabla(u^{n-\frac{\alpha}{2}}-u^{n,\alpha})\|&\leq \left(1-\frac{\alpha}{2}\right)\left(\frac{\alpha}{2}\right)\Delta t\int_{t_{n-1}}^{t_{n}}\|\nabla u_{tt}(s)\|ds\leq C\Delta t^2.\\
	\end{split}
	\end{equation}
	Using \eqref{erro1}, \eqref{dl3} and \eqref{erro2} in \eqref{dl2}, we get
	\begin{equation}\label{dl4}
	\langle ^RD^{\alpha}_{\Delta t}{\theta_{h}^n},\theta_{h}^{n,\alpha}\rangle\leq \left(\frac{2L+1}{2}\right)\|\theta_{h}^{n,\alpha}\|^2+C\big(\Delta t^{2}+h^2\big)^2.
	\end{equation}
	This gives us
	\begin{equation}\label{theta1}
	{^RD}^{\alpha}_{\Delta t}\|{\theta_{h}^n}\|^2\leq \big({2L+1}\big)\|\theta_{h}^{n,\alpha}\|^2+C\big(\Delta t^{2}+h^2\big)^2.
	\end{equation}
	From $(\ref{theta1}),$ one can get
	\begin{equation*}
	{^RD}^{\alpha}_{\Delta t}\|{\theta_{h}^n}\|^2\leq C_0\|\theta_{h}^{n,\alpha}\|^2+C\big(\Delta t^{2}+h^2\big)^2.
	\end{equation*}
	where $C_0=2L+1.$\\
	Further,
	\begin{equation}\label{A2}
	\begin{split}
	{^RD}^{\alpha}_{\Delta t}\|{\theta_{h}^n}\|^2&\leq C_0 \left(1-\frac{\alpha}{2}\right)^2\|\theta_{h}^n\|^2\\
	&+C_0 \left(\frac{\alpha}{2}\right)^2\|\theta_{h}^{n-1}\|^2
	+C\big(\Delta t^{2}+h^2\big)^2.
	\end{split}
	\end{equation}
	Using Lemma \ref{mainresult}, in \eqref{A2} one can find a positive constant ${\Delta t}^*$ such that when $\Delta t\leq {\Delta t}^*$, then
	\begin{equation*}
	\begin{split}
	\|{\theta_{h}^n}\|^2&\leq C\big(\Delta t^{2}+h^2\big)^2,\\
	\|{\theta_{h}^n}\|&\leq C\big(\Delta t^{2}+h^2\big).
	\end{split}
	\end{equation*}
	Now, an application of triangular inequality and Theorem \ref{projerr} completes the proof of Theorem \ref{fully1 theorem}.
\end{proof}
\section{Numerical experiments}
In this section, we present several examples to validate our theoretical findings.  The proposed scheme \eqref{fully discrete} is applied to solve all the considered problems.  The errors are calculated in $L^2(\Omega)$ norm at final time level $T=1$ and the convergence rate is computed as follows\\
\begin{equation*}
\begin{split}
rate=\left\{ \begin{array}{rcl}
\frac{\text{log}(e(\Delta t,\ h_1)/e(\Delta t,\ h_2))}{\text{log} (h_1/h_2)} &\mbox{in spatial direction,} \\
\frac{\text{log}(e(\Delta t_1,\ h)/e(\Delta t_2,\ h))}{\text{log}(\Delta t_1/\Delta t_2)} &\mbox{in temporal direction,}
\end{array}\right.\\
\end{split}
\end{equation*}
where $e(\Delta t, \ h_1)$, $e(\Delta t, \ h_2)$, $e(\Delta t_1, \ h)$, $e(\Delta t_2, \ h)$ are $L^2(\Omega)$ errors at stepsizes $h_1$, $h_2,$ $\Delta t_1$, $\Delta t_2$ $(h_1\neq h_2, \ \Delta t_1\neq \Delta t_2)$  respectively.
For all the problems we consider the time interval $[0,1]$  and $\alpha=0.4, 0.6$. In Newton's method, tolerance is taken to be $\epsilon=10^{-7}$ as a stopping criteria. For one-dimensional problems, we take the spatial domain $\Omega=[0,1]$ and for the two-dimensional problem, we take the spatial domain $\Omega =[0,1]\times[0,1]$.\\
\noindent \textbf{Example 1.} Consider the following one-dimensional time-fractional Fisher's equation 
\begin{equation}\label{ex1}
\begin{split}
^C{D}^{\alpha}_{t}u-\frac{\partial^2 u}{\partial x^2} &= f(u)+g_1(x,t),  \quad x\in \Omega,  \quad t\in(0,1],\\
u(x,t)&=0, \quad x\in \partial\Omega, \quad t\in(0,1], \\
u(x,0)&=0,\quad x\in \Omega,
\end{split}
\end{equation}
where
\begin{equation*}
\begin{split}
f(u)=&u(1-u),\\
\end{split}
\end{equation*}
and 
\begin{equation*}
g_1(x,t)=\frac{24t^{4-\alpha}}{\Gamma(5-\alpha)}\sin(2\pi x) +4\pi^2t^4\sin(2\pi x)-t^4\sin(2\pi x)(1-t^4\sin(2\pi x)).
\end{equation*}
The exact solution of problem \eqref{ex1} is given by
\begin{equation*}
u(x,t)=t^4 \sin(2\pi x).
\end{equation*}
Equation \eqref{ex1} has several applications in biology, chemical kinetics and flame propagation (\cite{D.,D.Li.}, and references therein).\\
To obtain the convergence rate in spatial direction for the problem \eqref{ex1}, we take $\Delta t= 10^{-3}$ for different values of $h$.  Similarly, to obtain the convergence rate in temporal direction, we take $ h=2\times 10^{-4}$ for different values of $\Delta t$. \\ 
$L^2(\Omega)$ errors and convergence rates in spatial direction for this example are given in Table \ref{tab1}. Similarly, $L^2(\Omega)$ errors and convergence rates in temporal direction are given in Table \ref{tab2}. It can be seen that for  $\alpha=0.4,\ 0.6$ these estimated convergence rates are tending to limit close  to $2$ which is in accordance with the theoretical convergence order.
\begin{table}[htbp]
	\centering
	\caption{$L^2(\Omega)$ errors and convergence rates in spatial direction for Example 1}
	\begin{tabular}{|r|r|r|r|r|}
		\hline
		\multicolumn{1}{|r}{} & \multicolumn{2}{|l}{\hspace{1.4cm}$\alpha=0.4$}       & \multicolumn{2}{|l|}{\hspace{1.4cm}$\alpha=0.6$}   \\
		\hline
		\multicolumn{1}{|l}{$h$} & \multicolumn{1}{|l}{$\|u^n-U^n\|_{L^2(\Omega)}$ }& \multicolumn{1}{|l|}{Rate} & \multicolumn{1}{l}{$\|u^n-U^n\|_{L^2(\Omega)}$}  & \multicolumn{1}{|l|}{Rate} \\
		\hline
		$\frac{1}{2^2}$ &2.2133e-1\hspace{0.5cm}    &   -    &  2.2030e-1 \hspace{0.4cm}     & - \\
		$\frac{1}{2^3}$ &5.7425e-2  \hspace{0.4cm}      &   1.9465    & 5.7033e-2 \hspace{0.4cm}     & 1.9496 \\
		$\frac{1}{2^4}$ &1.4472e-2 \hspace{0.4cm}      &  1.9884    & 1.4365e-2 \hspace{0.4cm}     & 1.9892 \\
		$\frac{1}{2^5}$ & 3.6255e-3 \hspace{0.4cm}     &   1.9970    & 3.5983e-3 \hspace{0.4cm}     & 1.9972 \\
		\hline
	\end{tabular}%
	\label{tab1}%
\end{table}%

\begin{table}[htbp]
	\centering
	\caption{$L^2(\Omega)$ errors and convergence rates in temporal direction for Example 1}
	\begin{tabular}{|r|r|r|r|r|}
		\hline
		\multicolumn{1}{|r}{} & \multicolumn{2}{|l}{\hspace{1.4cm}$\alpha=0.4$}       & \multicolumn{2}{|l|}{\hspace{1.4cm}$\alpha=0.6$}   \\
		\hline
		\multicolumn{1}{|l}{$\Delta t$} & \multicolumn{1}{|l}{$\|u^n-U^n\|_{L^2(\Omega)}$ }& \multicolumn{1}{|l|}{Rate} & \multicolumn{1}{l}{$\|u^n-U^n\|_{L^2(\Omega)}$}  & \multicolumn{1}{|l|}{Rate} \\
		\hline
		$\frac{1}{2^2}$ &3.6934e-2\hspace{0.5cm}    &   -    &  4.9862e-2 \hspace{0.4cm}     & - \\
		$\frac{1}{2^3}$ &9.7834e-3 \hspace{0.4cm}      &   1.9165   & 1.3009e-2 \hspace{0.4cm}     & 1.9384 \\
		$\frac{1}{2^4}$ &2.5125e-3 \hspace{0.4cm}      &  1.9612    & 3.3145e-3 \hspace{0.4cm}     & 1.9726 \\
		$\frac{1}{2^5}$ & 6.3642e-4 \hspace{0.4cm}     &   1.9811   & 8.3614e-4 \hspace{0.4cm}     & 1.9870 \\
		\hline
	\end{tabular}%
	\label{tab2}%
\end{table}%
\noindent \textbf{Example 2}. We consider the following two-dimensional time-fractional Huxley equation
\begin{equation}\label{exscalr1}
\begin{split}
^C{D}^{\alpha}_{t}u(x,t)-\Delta u(x,t) &= f(u)+g_2(x,t),  \quad x\in \Omega,  \quad t\in(0,1],\\
u(x,t)&= 0, \quad x\in \partial\Omega, \quad t\in(0,1], \\
u(x,0)&=0,\hspace{.3cm} \quad x\in \Omega,
\end{split}
\end{equation}
where
\begin{equation*}
\begin{split}
f(u)=&u(1-u)(u-1),\\
\end{split}
\end{equation*}
and 
\begin{equation*}
\begin{split}
g_2(x,t)=\frac{6t^{3-\alpha}}{\Gamma(4-\alpha)}&(1-x_1)\sin(x_1)(1-x_2)\sin(x_2)+2t^3\Big((1-x_1)\sin(x_1)(1-x_2)\sin(x_2)\\
&+\cos(x_1)(1-x_2)\sin(x_2)
+(1-x_1)\sin(x_1)\cos(x_2)\Big)\\
+t^3&(1-x_1)\sin(x_1)(1-x_2)\sin(x_2)\big(t^3(1-x_1)\sin(x_1)(1-x_2)\sin(x_2)-1\big)^2.\\
\end{split}
\end{equation*}
The exact solution of problem \eqref{exscalr1} is given by
\begin{equation*}
u=t^3(1-x_1)\sin(x_1)(1-x_2)\sin(x_2),
\end{equation*}
where $x=(x_1,x_2)\in [0,1]\times [0,1]$.\\
Equation \eqref{exscalr1} describes numerous physical models in areas like population genetics in circuit theory and the transmission of nerve impulses (\cite{D.,D.Li.} and references therein).\\
To obtain the convergence rate in spatial direction for the considered two-dimensional problem, we take $\Delta t= 10^{-3}$ for different values of $h$. Similarly, to obtain the convergence rate in temporal direction, we take $ h=\frac{1}{525}$ for different values of $\Delta t$. \\ 
\noindent In this example, $L^2(\Omega)$ errors and convergence rates in spatial direction are given in Table \ref{tab5}. Similarly, $L^2(\Omega)$ errors and convergence rates in temporal direction are given in Table \ref{tab6}. Again, it can be seen that these estimated convergence rates are tending to limit close to $2$ which is in accordance with the theoretical convergence order.
\begin{table}[htbp]
	\centering
	\caption{$L^2(\Omega)$ errors and convergence rates in spatial direction for Example 2}
	\begin{tabular}{|r|r|r|r|r|}
		\hline
		\multicolumn{1}{|r}{} & \multicolumn{2}{|l}{\hspace{1.4cm}$\alpha=0.4$}       & \multicolumn{2}{|l|}{\hspace{1.4cm}$\alpha=0.6$}   \\
		\hline
		\multicolumn{1}{|l}{$h$} & \multicolumn{1}{|l}{$\|u^n-U^n\|_{L^2(\Omega)}$ }& \multicolumn{1}{|l|}{Rate} & \multicolumn{1}{l}{$\|u^n-U^n\|_{L^2(\Omega)}$}  & \multicolumn{1}{|l|}{Rate} \\
		\hline
		$\frac{1}{2^2}$ &5.3226e-3\hspace{0.5cm}    &   -    &  5.2790e-3 \hspace{0.4cm}     & - \\
		$\frac{1}{2^3}$ &1.3998e-3 \hspace{0.4cm}      &   1.9269   & 1.3860e-3 \hspace{0.4cm}     & 1.9294 \\
		$\frac{1}{2^4}$ &3.5447e-4 \hspace{0.4cm}      &  1.9815    & 3.5079e-4 \hspace{0.4cm}     & 1.9822 \\
		$\frac{1}{2^5}$ & 8.8910e-5 \hspace{0.4cm}     &   1.9952   & 8.7980e-5 \hspace{0.4cm}     & 1.9954 \\
		\hline
	\end{tabular}%
	\label{tab5}%
\end{table}%

\begin{table}[htbp]
	\centering
	\caption{$L^2(\Omega)$ errors and convergence rates in temporal direction for Example 2}
	\begin{tabular}{|r|r|r|r|r|}
		\hline
		\multicolumn{1}{|r}{} & \multicolumn{2}{|l}{\hspace{1.4cm}$\alpha=0.4$}       & \multicolumn{2}{|l|}{\hspace{1.4cm}$\alpha=0.6$}   \\
		\hline
		\multicolumn{1}{|l}{$\Delta t$} & \multicolumn{1}{|l}{$\|u^n-U^n\|_{L^2(\Omega)}$ }& \multicolumn{1}{|l|}{Rate} & \multicolumn{1}{l}{$\|u^n-U^n\|_{L^2(\Omega)}$}  & \multicolumn{1}{|l|}{Rate} \\
		\hline
		$\frac{1}{2}$\hspace{.1cm} &3.1544e-3\hspace{0.5cm}    &   -    &  4.3171e-3 \hspace{0.4cm}     & - \\
		$\frac{1}{2^2}$ &8.2884e-4 \hspace{0.4cm}      &   1.9282    & 1.1080e-3 \hspace{0.4cm}     & 1.9621 \\
		$\frac{1}{2^3}$ &2.1292e-4 \hspace{0.4cm}      &  1.9608   & 2.8199e-4 \hspace{0.4cm}     & 1.9743 \\
		$\frac{1}{2^4}$ & 5.4130e-5 \hspace{0.4cm}     &   1.9758  & 7.1315e-5 \hspace{0.4cm}     & 1.9833 \\
		\hline
	\end{tabular}%
	\label{tab6}%
\end{table}%
\par Note that we derived the error estimate in $L^2(\Omega)$ norm under the regularity assumption and certain compatibility conditions (see Remark \ref{remark}) on solution $u$ of the problem \eqref{cuc:1.1}-\eqref{cuc:1.4}. In general, regularity of initial data does not guarantee the regularity of the exact solution for the time-fractional partial differential equation. So straightforward implementation of scheme \eqref{fully discrete} does not achieve $O(\Delta t^2)$ accuracy \cite{jin2017analysis}. We demonstrate this by taking the following example in which the scheme \eqref{fully discrete} achieves $O(\Delta t)$ accuracy.\\
\noindent \textbf{Example 3}. Consider the following one-dimensional problem with unknown exact solution
\begin{equation}\label{ex2}
\begin{split}
^C{D}^{\alpha}_{t}u-\frac{\partial^2 u}{\partial x^2} &= f(u),  \quad x\in \Omega,  \quad t\in(0,1],\\
u(x,t)&=0, \quad x\in \partial\Omega, \quad t\in(0,1], \\
u(x,0)&=0,  \quad x\in \Omega,
\end{split}
\end{equation}
with
\begin{equation*}
\begin{split}
f(u)=&5+u(1+u^3).\\
\end{split}
\end{equation*}
Since the exact solution to the problem \eqref{ex2} is not known, we compute reference solution $u_{ref}(t) $ with finer mesh $h_{ref}=2^{-8}$ and $\Delta t_{ref}=2^{-10}$. To obtain the convergence rate in spatial direction, we compute the numerical solution $U_h^n$ with $\Delta t=2^{-10}$ for different values of $h$. Similarly, to obtain the convergence rate in temporal direction, we compute numerical solution $U_h^n$ with $h=2^{-8}$ for different values of $\Delta t$. The $L^2(\Omega)$ errors and convergence rates for Example $3$ in spatial direction and temporal direction are given in Tables \ref{tab3} and \ref{tab4} respectively.
\begin{table}[htbp]
	\centering
	\caption{$L^2(\Omega)$ errors and convergence rates in spatial direction for Example 3}
	\begin{tabular}{|r|r|r|r|r|}
		\hline
		\multicolumn{1}{|r}{} & \multicolumn{2}{|l}{\hspace{1.4cm}$\alpha=0.4$}       & \multicolumn{2}{|l|}{\hspace{1.4cm}$\alpha=0.6$}   \\
		\hline
		\multicolumn{1}{|l}{$h$} & \multicolumn{1}{|l}{$\|u^n-U^n\|_{L^2(\Omega)}$ }& \multicolumn{1}{|l|}{Rate} & \multicolumn{1}{l}{$\|u^n-U^n\|_{L^2(\Omega)}$}  & \multicolumn{1}{|l|}{Rate} \\
		\hline
		$\frac{1}{2^2}$ &5.2126e-3\hspace{0.5cm}    &   -    &  5.0108e-3 \hspace{0.4cm}     & - \\
		$\frac{1}{2^3}$ &1.3412e-3 \hspace{0.4cm}      &   1.9585    & 1.2902e-3 \hspace{0.4cm}     & 1.9575 \\
		$\frac{1}{2^4}$ &3.3647e-4 \hspace{0.4cm}      &  1.9949    & 3.2370e-4 \hspace{0.4cm}     & 1.9948 \\
		$\frac{1}{2^5}$ & 8.3254e-5 \hspace{0.4cm}     &   2.0149   & 8.0095e-5 \hspace{0.4cm}     & 2.0149 \\
		\hline
	\end{tabular}%
	\label{tab3}%
\end{table}%

\begin{table}[htbp]
	\centering
	\caption{$L^2(\Omega)$ errors and convergence rates in temporal direction for Example 3}
	\begin{tabular}{|r|r|r|r|r|}
		\hline
		\multicolumn{1}{|r}{} & \multicolumn{2}{|l}{\hspace{1.4cm}$\alpha=0.4$}       & \multicolumn{2}{|l|}{\hspace{1.4cm}$\alpha=0.6$}   \\
		\hline
		\multicolumn{1}{|l}{$\Delta t$} & \multicolumn{1}{|l}{$\|u^n-U^n\|_{L^2(\Omega)}$ }& \multicolumn{1}{|l|}{Rate} & \multicolumn{1}{l}{$\|u^n-U^n\|_{L^2(\Omega)}$}  & \multicolumn{1}{|l|}{Rate} \\
		\hline
		$\frac{1}{2^3}$ &5.3833e-4\hspace{0.5cm}    &   -    &  3.5785e-4 \hspace{0.4cm}     & - \\
		$\frac{1}{2^4}$ &2.8059e-4 \hspace{0.4cm}      &   0.9400    & 2.0541e-4 \hspace{0.4cm}     & 0.8009 \\
		$\frac{1}{2^5}$ &1.4107e-4 \hspace{0.4cm}      &  0.9921   & 1.0762e-4 \hspace{0.4cm}     & 0.9325 \\
		$\frac{1}{2^6}$ & 6.8939e-5 \hspace{0.4cm}     &   1.0330   & 5.3618e-5 \hspace{0.4cm}     & 1.0052 \\
		\hline
	\end{tabular}%
	\label{tab4}%
\end{table}%
\section{Conclusions}In this paper, we have proposed Crank-Nicolson-Galerkin finite element scheme for solving  the time-fractional nonlinear diffusion equation by using Newton's method. Well-posedness results have been discussed at discrete level. We established the fractional Gronwall type inequality for the Gr$\mathrm{\ddot{u}}$nwald-Letnikov approximation to the Riemann-Liouville fractional derivative. The efficiency of the proposed scheme has been demonstrated by several numerical experiments. For future work, it is worth exploring whether $O(\Delta t^2)$ accuracy is preserved by relaxing the regularity assumption and compatibility conditions which are considered in this paper. 
\section*{Acknowledgment}
First author is thankful to University Grants Commission, India, for financial grant through Senior Research  Fellowship.

\bibliographystyle{plain}
\bibliography{nonlinear}
\end{document}